\documentclass[11pt]{article}

\usepackage{oldlfont}
\usepackage{amsfonts}
\usepackage{graphicx}
\usepackage{eepic}

\newcommand{\be}{\begin{equation}}
\newcommand{\ee}{\end{equation}}
\newcommand{\bea}{\begin{eqnarray}}
\newcommand{\eea}{\end{eqnarray}}
\newcommand{\beas}{\begin{eqnarray*}}
\newcommand{\eeas}{\end{eqnarray*}}
\newcommand{\ba}{\begin{array}}
\newcommand{\ea}{\end{array}}

\newcommand{\<}  {\langle}
\renewcommand{\>}{\rangle}
\newcommand{\field}[1]{\mathbb{#1}}
\newcommand{\supp}{{\mathrm supp}}
\newcommand{\diam}{{\mathrm diam}}

\renewcommand{\div}{{\mathrm div}}

\newcommand{\R}{\mbox{\rm I\kern-.18em R}}
\newcommand{\G}{\Gamma}

\newcommand{\g}{\gamma}
\newcommand{\eps}{\varepsilon}

\newcommand{\bu}{{\mathbf u}}
\newcommand{\bv}{{\mathbf v}}

\newcommand{\bx}{{\mathbf x}}

\newcommand{\boldb}{{\mathbf b}}
\newcommand{\boldf}{{\mathbf f}}

\newcommand{\bX}{{\mathbf X}}
\newcommand{\bH}{{\mathbf H}}
\newcommand{\bL}{{\mathbf L}}
\newcommand{\bcurl}{{\mathbf curl}}
\newcommand{\bzero}{{\mathbf 0}}

\newcommand{\divg}{{\mathrm div}_{\G}}

\newcommand{\bcurlg}{{\mathbf curl}_{\G}}

\newcommand{\bxi}{\hbox{\mathversion{bold}$\xi$}}
\newcommand{\bPsi}{\hbox{\mathversion{bold}$\Psi$}}
\newcommand{\bpsi}{\hbox{\mathversion{bold}$\psi$}}
\newcommand{\bchi}{\hbox{\mathversion{bold}$\chi$}}

\newcommand{\bphi}{\hbox{\mathversion{bold}$\varphi$}}

\newcommand{\CM}{{\cal M}}

\newcommand{\CP}{{\cal P}}
\newcommand{\bCP}{\hbox{\mathversion{bold}$\cal P$}}

\newcommand{\CT}{{\cal T}}

\newtheorem{theorem}{Theorem} [section]
\newtheorem{lemma}{Lemma} [section]

\newtheorem{prop}{Proposition} [section]
\newtheorem{remark}{Remark} [section]
\newenvironment{proof}{\noindent\textbf{Proof.}\ }
              {\nopagebreak\hbox{ }\hfill$\Box$\bigskip}
\newcommand{\qed}{\nopagebreak\hbox{ }\hfill$\Box$\bigskip}

\title{Natural $hp$-BEM for the electric field integral equation with singular
       solutions
\thanks{Supported by EPSRC under grant no. EP/E058094/1.}}

\author{Alexei Bespalov
\thanks{School of Mathematics, University of Manchester,
        Manchester, M13 9PL, UK.
        Email: {\tt albespalov@yahoo.com}
        }
        \and
        Norbert Heuer
\thanks{Facultad de Matem\'aticas, Pontificia Universidad Cat\'olica de Chile,
        Avenida Vicu\~na Mackenna 4860, Santiago, Chile.
        Email: {\tt nheuer@mat.puc.cl}}
        }
\begin{document}
\date{}
\maketitle

\begin{abstract}
We apply the $hp$-version of the boundary element
method (BEM) for the numerical solution of the electric field
integral equation (EFIE) on a Lipschitz polyhedral surface $\G$.
The underlying meshes are supposed to be quasi-uniform triangulations of $\G$,
and the approximations are based on either Raviart-Thomas or Brezzi-Douglas-Marini
families of surface elements.
Non-smoothness of $\Gamma$ leads to singularities in the solution
of the EFIE, severely affecting convergence rates of the BEM.
However, the singular behaviour of the solution can be explicitly
specified using a finite set of power functions (vertex-, edge-, and
vertex-edge singularities). In this paper we use this fact to
perform an a priori error analysis of the $hp$-BEM on quasi-uniform meshes.
We prove precise error estimates in terms of the polynomial degree $p$,
the mesh size $h$, and the singularity exponents.


\end{abstract}

%%%%%%%%%%%%%%%%%%%%%%%%%%%%%%%%%%%%%%%%%%%%%%%%%%%%%%%%%%%%%%%%%%%%%%%%%%%%%%%%
\section{Introduction} \label{sec_intro}
\setcounter{equation}{0}

In this paper we study numerical approximations of the electric field integral equation (EFIE)
on a surface $\Gamma = \partial\Omega$, where $\Omega \subset {\field{R}}^3$ is
a Lipschitz polyhedron. The EFIE is a boundary integral formulation of a
boundary value problem for the time-harmonic Maxwell equations
in the domain exterior to $\Omega$.
It models the scattering of electromagnetic waves at a perfectly conducting body.

For the numerical solution of the EFIE we use
the Galerkin boundary element method (BEM).
It employs $\bH(\divg)$-conforming families of surface elements,
namely Raviart-Thomas (RT) and Brezzi-Douglas-Marini (BDM) elements, to discretise
the variational formulation of the EFIE (called Rumsey's principle).
This approach is referred to as the natural BEM for the EFIE.
As in finite element methods, the convergence may be achieved either
by keeping polynomial degrees fixed and refining the mesh ($h$-version),
or by fixing the mesh and increasing polynomial degrees ($p$-version),
or by simultaneous $h$-refinement and $p$-enrichment ($hp$-version).

The Galerkin BEM is widely used in the engineering practice for the
simulation of electromagnetic scattering. Moreover, it had been used
long before a rigorous theoretical analysis of the method became available.
Error analyses of different boundary element schemes for the eddy current problem on smooth obstacles are given by MacCamy and Stephan in \cite{McCamyS_83_BET,McCamyS_84_SPT}.
Despite these early results the BEM-analysis for non-smooth obstacles started much later.
In fact, the convergence and a priori error analysis of the $h$-BEM for the EFIE
on piecewise smooth (open or closed) surfaces
has been developed within the last decade
(see \cite{BuffaCS_02_BEM,HiptmairS_02_NBE,BuffaHvPS_03_BEM,BuffaC_03_EFI,BuffaH_03_GBE}),
and the corresponding results for high-order methods ($p$- and $hp$-BEM)
are very recent (see \cite{BespalovH_NpB,BespalovH_10_hpA,BespalovHH_Chp}).
In particular, \cite{BespalovH_NpB} and \cite{BespalovHH_Chp} establish quasi-optimal convergence
of the natural $hp$-BEM on meshes of shape-regular elements.
An essential ingredient for the proofs in these papers are the projection-based
interpolation operators developed by Demkowicz and co-authors
\cite{DemkowiczB_03_pIE,DemkowiczB_05_HHH,Demkowicz_08_PES}.
These operators also facilitate a priori error analysis, where one
needs $hp$-approximation theory in specific trace spaces.
In \cite{BespalovH_10_hpA} we prove an a priori error estimate
of the $hp$-BEM on quasi-uniform meshes of affine elements under the assumption
that the regularity of the solution to the EFIE is given in Sobolev
spaces on $\G$. The latter result states that the 
method converges with the same rate $r + \frac 12$ in both $h$ and $p^{-1}$ (here,
$r$ denotes the Sobolev regularity order, and $p$ is assumed to be large enough).
However, it is evident from the numerical results reported in \cite{Leydecker_Thesis} that
the $p$-BEM converges faster than the $h$-BEM for the EFIE.
This is similar to what was observed and proved for the $p$-BEM and the $h$-BEM
applied to elliptic problems in three-dimensions (cf.
\cite{SchwabS_96_OpA,HeuerMS_99_ECB,BespalovH_08_hpB,BespalovH_10_hpW}).

With this paper we fill a gap in the theory of the BEM for the EFIE on polyhedral
surfaces by proving a precise error estimate for the $hp$-BEM on quasi-uniform meshes.
Similarly to the elliptic case we make use of explicit expressions for singularities in
electromagnetics fields (these expressions are available from
\cite{CostabelD_00_SEF,BespalovH_NpB}). The established error
estimate shows that the convergence
rates in $h$ and $p$ depend on the strongest singularity exponent, and that the $p$-BEM
converges twice as fast as the $h$-BEM on quasi-uniform meshes.
This extends the results of \cite{BespalovH_NpB}, where the
analysis was restricted to the $p$-BEM on a plane open surface.

It is now well known that appropriate decompositions of vector fields
(on both the continuous and discrete level) are
critical for the analysis of electromagnetic problems and their approximations.
In the case of the EFIE the main idea is to isolate the kernel of the $\divg$-operator
such that the complementary field possesses an enhanced smoothness
(cf. \cite{BuffaC_03_EFI,Buffa_05_RDS,BespalovHH_Chp}).
The corresponding decompositions of singular vector fields greatly facilitate
our analysis in this paper as well. In particular, the complementary vector fields are singular
vector functions belonging to $\bH^{1/2}(\G)$. Then, in the case of the BDM-based BEM,
these vector functions can be approximated component-wise by continuous piecewise polynomials,
and the desired $hp$-error estimates are derived by using the corresponding
results in \cite{BespalovH_08_hpB} for scalar singularities (belonging to $H^{1/2}(\G)$).
However, this simple approach does not work for the RT-based BEM because
the dimension of the underlying RT-space on the reference element is smaller
than the dimension of the BDM-counterpart. That is why, the results
of \cite{BespalovH_08_hpB} cannot be applied directly, and additional technical
arguments are needed (see the proof of Lemma~\ref{lm_approx_v}
in Section~\ref{sec_proof_v}).

The rest of the paper is organised as follows.
In the next section we formulate the EFIE in its variational form
and recall the typical structure of the solution to this model problem.
In Section~\ref{sec_BEM} we introduce the $hp$-version of the BEM
for the EFIE and formulate the main result of the paper (Theorem~\ref{thm_main_hp})
stating convergence rates of the method. This result follows from
the general approximation theorem (Theorem~\ref{thm_gen_approx}) established
in Section~\ref{sec_gen_approx}. The proof of Theorem~\ref{thm_gen_approx}
relies, in particular, on two technical lemmas which are proved in
Section~\ref{sec_proofs}.

Throughout the paper, $C$ denotes a generic positive constant
independent of $h$ and $p$.

%%%%%%%%%%%%%%%%%%%%%%%%%%%%%%%%%%%%%%%%%%%%%%%%%%%%%%%%%%%%%%%%%%%%%%%%%%%%%%%%
\section{Formulation of the problem} \label{sec_problem}
\setcounter{equation}{0}

Let $\G$ be a Lipschitz polyhedral surface in ${\field{R}}^3$.
Throughout the paper we will use exactly the same
notation as in \cite{BespalovH_10_hpA} for all involved differential and
boundary integral operators as well as for Sobolev spaces of scalar functions
and tangential vector fields on $\G$ (all essential definitions are given
in \cite[Section~3.1]{BespalovH_10_hpA}). In particular, we use boldface symbols
for vector fields, and the spaces (or sets) of vector fields are denoted in
boldface as well.

For a given wave number $\kappa > 0$, we denote by $\Psi_\kappa$ (resp., $\bPsi_\kappa$)
the scalar (resp., the vectorial) single layer boundary integral operator
on $\G$ for the Helmholtz operator $-\,\Delta - \kappa^2$
(see \cite[Section~5]{BuffaH_03_GBE}).
The variational formulation for the EFIE will be posed in the following
Hilbert space of tangential vector fields on $\G$:
\[
  \bX = \bH^{-1/2}(\divg,\G) :=
        \{\bu \in \bH^{-1/2}_{\|}(\G);\; \divg\,\bu \in H^{-1/2}(\G)\},
\]
which is the trace space of $\bH(\bcurl,\Omega)$, where $\Omega \subset {\field{R}}^3$
is a Lipschitz polyhedron such that $\G = \partial\Omega$
(we refer to \cite{BuffaC_01_TFI,BuffaC_01_TII,BuffaCS_02_THL}
for the definition and properties of this and other trace spaces on $\G$).

Let $\bX'$ be the dual space of $\bX$ (with duality pairing
extending the $\bL^2(\G)$-inner product for tangential vector fields).
Then, for a given source functional $\boldf\in \bX'$
%% ($\boldf$ represents the excitation by an incident wave),
the variational formulation of the EFIE reads as:
{\em find a complex tangential field $\bu\in\bX$ such that}
\be \label{bie_var}
    a(\bu,\bv)
    :=
    \<\Psi_\kappa \divg\,\bu, \divg\,\bv\> - \kappa^2 \<\bPsi_\kappa\bu, \bv\>
    = \<\boldf, \bv\> \quad\forall \bv \in \bX.
\ee
Here, the brackets $\<\cdot,\cdot\>$ denote
dualities associated with $H^{1/2}(\G)$ and $\bH^{1/2}_{\|}(\G)$.
To ensure the uniqueness of the solution to (\ref{bie_var}) we always assume
that $\kappa^2$ is not an electrical eigenvalue of the interior problem.

Let us recall the typical structure of the solution $\bu$ to problem (\ref{bie_var}),
provided that the source functional $\boldf$ is sufficiently smooth
(we note that this regularity assumption is satisfied for the electromagnetic
scattering with plane incident wave). We use the results of
\cite[Appendix~A]{BespalovH_NpB}. These results were derived from the regularity
theory in \cite{CostabelD_00_SEF} for the boundary value problems
for Maxwell's equations in 3D by making use of trace arguments and some
technical calculations in the spirit of \cite{vonPetersdorffS_90_DEC} and \cite{vP}.

Let $V = \{v\}$ and $E = \{e\}$ denote the sets of vertices and edges of $\G$, respectively.
For $v\in V$, let $E(v)$ denote the set of edges with $v$ as an end point.
Then the solution $\bu$ of (\ref{bie_var}) can be written as
\be \label{dec}
    \bu = \bu_{\rm reg} + \bu_{\rm sing},
\ee
where
\be \label{reg}
    \bu_{\rm reg} \in
%%  \bH^k_{\;-}(\divg,\G)
    \bX^k :=
    \{\bu \in \bH^k_{\;-}(\G);\; \divg\,\bu \in H^k_{\;-}(\G)\}\ \ \hbox{with \ $k>0$}
\ee
(the spaces $\bH^k_{\;-}(\G)$ and $H^k_{\;-}(\G)$ are
defined in a piecewise fashion by localisation to each face of $\G$,
and the space $\bX^k$ is equipped with its graph norm $\|\cdot\|_{\bX^k}$,
see \cite{BespalovH_10_hpA}),
\be \label{sing}
    \bu_{\rm sing} = \sum_{e\in E} \bu^e + \sum_{v\in V} \bu^v
                     + \sum_{v\in V}\sum_{e\in E(v)} \bu^{ev},
\ee
and $\bu^e$, $\bu^v$, and $\bu^{ev}$ are the edge, vertex, and
edge-vertex singularities, respectively.

In order to write explicit expressions for the above singularities,
let us fix a vertex $v \in V$ and an edge $e \in E(v)$.
Then, on each face $\G^{ev} \subset \G$ such that $e \subset \partial\G^{ev}$
we will use local polar and Cartesian coordinate systems $(r_v,\theta_v)$
and $(x_{e1},x_{e2})$, both with the origin at $v$, such that
$e = \{(x_{e1},x_{e2});\; x_{e2} = 0,\ x_{e1} > 0\}$ and
for a sufficiently small neighbourhood $B_\tau$ of $v$ there holds
$\G^{ev} \cap B_\tau \subset \{(r_v,\theta_v);\; 0 < \theta_v < \omega_v\}$.
Here, $\omega_v$ denotes the interior angle (on $\G^{ev}$)
between the edges meeting at $v$.
For simplicity of notation we write out here only the leading singularities
in $\bu^e$, $\bu^v$, and $\bu^{ev}$ on the face $\G^{ev}$,
thus omitting the corresponding terms of higher regularity
(see \cite[Appendix~A]{BespalovH_NpB} for complete expansions).

For the edge singularities $\bu^{e}$ one has
\be \label{ue}
    \bu^e =
            \bcurl_{\G^{ev}}\Big(x_{e2}^{\g_1^e}\,|\log x_{e2}|^{s_1^e}\,b_{1}^e(x_{e1})\,
            \chi_1^e(x_{e1})\,\chi_2^e(x_{e2})\Big)
            +\,
            x_{e2}^{\g_2^e}\,|\log x_{e2}|^{s_2^e}\,\boldb_{2}^e(x_{e1})\,
            \chi_1^e(x_{e1})\,\chi_2^e(x_{e2}),
\ee
where $\bcurl_{\G^{ev}} = (\partial/\partial x_{e2},\, -\partial/\partial x_{e1})$
is the tangential vector curl operator $\bcurlg$ restricted to the face $\G^{ev}$
(cf. \cite{BuffaC_01_TFI,BuffaC_01_TII}),
$\gamma_{1}^e,\;\gamma_{2}^e > \frac 12$, and $s_1^e,\,s_2^e\ge 0$ are
integers. Here, $\chi_1^e$, $\chi_2^e$ are $C^\infty$ cut-off functions
with $\chi_1^e=1$ in a certain distance to the end points of $e$
and $\chi_1^e=0$ in a neighbourhood of these vertices.
Moreover, $\chi_2^e=1$ for $0\le x_{e2}\le\delta_e$
and $\chi_2^e=0$ for $x_{e2}\ge 2\delta_e$ with some $\delta_e \in (0,\frac 12)$.
The functions $b_{1}^e\chi_1^e \in H^{m_1}(e)$ and $\boldb_{2}^e\chi_1^e \in \bH^{m_2}(e)$
for $m_1$ and $m_2$ as large as required.

The vertex singularities $\bu^v$ have the form
\be \label{uv}
    \bu^v =
    \bcurl_{\G^{ev}}\Big(r_v^{\lambda_1^{v}}\, |\log r_v|^{q_1^v}\,
    \chi^v(r_v)\, \chi^v_{1}(\theta_v)\Big)
    +\,r_v^{\lambda_2^{v}}\, |\log r_v|^{q_2^v}\,
    \chi^v(r_v)\, \bchi^v_{2}(\theta_v),
\ee
where $\lambda_1^v,\;\lambda_2^v > - \frac 12$ are real numbers,
$q_1^v,\,q_2^v \ge 0$ are integers,
$\chi^v$ is a $C^\infty$ cut-off function with
$\chi^v=1$ for $0\le r_v\le\tau_v$ and $\chi^v=0$ for $r_v\ge 2\tau_v$
with some $\tau_v \in (0,\frac 12)$.
The functions $\chi_1^v,\,\bchi_2^v$ are such that
$\chi_1^v \in H^{t_1}(0,\omega_v)$,
$\bchi_2^v \in \bH^{t_2}(0,\omega_v)$
for $t_1$, $t_2$ as large as required.

For the combined edge-vertex singularity $\bu^{ev}$ one has
\[
    \bu^{ev} = \bu_1^{ev} + \bu_2^{ev},
\]
where
\bea \label{u1ev}
     \bu_1^{ev}
     & = &
     \bcurl_{\G^{ev}}\Big(
     x_{e1}^{\lambda_1^v-\gamma_1^e} x_{e2}^{\gamma_1^e}
     |\log x_{e1}|^{\beta_1}|\log x_{e2}|^{\beta_2}\,
     \chi^v(r_v) \chi^{ev}(\theta_v)
     \Big)
     \cr\cr
     & &
     \qquad\qquad +\,
     x_{e1}^{\lambda_2^v-\gamma_2^e} x_{e2}^{\gamma_2^e}
     |\log x_{e1}|^{\beta_3}|\log x_{e2}|^{\beta_4}\,
     \chi^v(r_v) \chi^{ev}(\theta_v)
     \left(
     \ba{c}
     0 \\
     1
     \ea
     \right)
\eea
and
\be \label{u2ev}
    \bu_2^{ev} =
    \bcurl_{\G^{ev}}\Big(x_{e2}^{\g_1^e}\,|\log x_{e2}|^{s_1^e}\,b_{3}^e(x_{e1},x_{e2})\,
    \chi_2^e(x_{e2})\Big)
    +\,
    x_{e2}^{\g_2^e}\,|\log x_{e2}|^{s_2^e}\,\boldb_{4}^e(x_{e1},x_{e2})\,
    \chi_2^e(x_{e2}).
\ee
Here, $\lambda_i^v$, $\gamma_i^e$, $s_i^e$ ($i=1,2$), $\chi^v$, and $\chi_2^{e}$
are as above, $\beta_k \ge 0$ ($k = 1\ldots,4$) are integers,
$\beta_1 +\beta_2 = s_1^e +q_1^v$, $\beta_3 +\beta_4 = s_2^e +q_2^v$
with $q_1^v$, $q_2^v$ being as in (\ref{uv}),
$\chi^{ev}$ is a $C^\infty$ cut-off function with
$\chi^{ev}=1$ for $0\le\theta_v\le\beta_v$ and $\chi^{ev}=0$ for
$\frac 32\beta_v\le\theta_v\le\omega_v$ for some
$\beta_v \in (0,\min\{\omega_v/2,\pi/8\}]$.
The functions $b_{3}^e$ and $\boldb_{4}^e$, when extended by zero onto
$\R^{2+}:=\{(x_{e1},x_{e2});\; x_{e2}>0\}$, lie in $H^{m_1}(\R^{2+})$ and
$\bH^{m_2}(\R^{2+})$, respectively, with $m_1$, $m_2$ as large as required.
Finally, the supports of $\bu_1^{ev}$ and $\bu_2^{ev}$ are subsets of the sector
$\bar S_{ev} = \{(r_v,\theta_v);\;
                 0 \le r_v \le 2\tau_v,\, 0 \le \theta_v \le \frac 32 \beta_v\}$.

\begin{remark} \label{rem_sing1}
{\rm (i)}
The exponents $\g_i^e$ ($i=1,2$) of the edge and vertex-edge singularities
in {\rm (\ref{ue})}, {\rm (\ref{u1ev})}, {\rm (\ref{u2ev})} satisfy 
$\gamma_{i}^e > \frac 12$. However, for our approximation analysis below
it suffices to require that $\gamma_{i}^e > 0$ ($i=1,2$).
Note that $\gamma_{i}^e > 0$ and $\lambda_{i}^v > -\frac 12$ ($i=1,2$)
are the minimum requirements to guarantee $\bu \in \bX$.

{\rm (ii)}
As mentioned above, the terms of higher regularity (i.e., with greater
singularity exponents) are omitted in {\rm (\ref{ue})--(\ref{u2ev})}.
These terms are necessary
to obtain the regular part $\bu_{\rm reg} \in \bX^k$ of decomposition {\rm (\ref{dec})}
as smooth as required. This can be done by considering sufficiently many
(omitted) singularity terms of each type.
\end{remark}

\begin{remark} \label{rem_sing2}
{\rm (i)}
By {\rm (\ref{ue})--(\ref{u2ev})} we conclude that any singular vector
field $\bu^s$ in {\rm (\ref{sing})} ($s = e,\, v$, or $ev$) can be written as
\be \label{us}
    \bu^s = \bcurlg\, w^s + \bv^s = \bcurlg\, w^s + (v^s_1,\,v^s_2)
\ee
with corresponding (scalar) singular functions $w^s,\,v^s_1,\,v^s_2$
being defined on the whole surface $\G$. Note that these
scalar functions are $H^{1/2}$-regular on $\G$, i.e.,
\[
  w^s \in H^{1/2}(\G),\qquad \bv^s = (v^s_1,\,v^s_2) \in \bH^{1/2}(\G),\qquad
  s = e,\, v,\, ev.
\]

{\rm (ii)}
It is important to observe that the functions $w^s,\,v_1^s,\,v_2^s$
($s = e,\, v$, or $ev$) in {\rm (\ref{us})} are typical scalar singularities
inherent to solutions of the boundary integral equations
with hypersingular operator for the Laplacian on $\G$
and with possibly singular right-hand side.
Continuous piecewise polynomial approximations of these scalar singularities
in fractional-order Sobolev spaces were analysed
in~{\rm \cite{BespalovH_08_hpB,Bespalov_09_NPA}}
and will be used to prove the main result of the present paper.
\end{remark}

%%%%%%%%%%%%%%%%%%%%%%%%%%%%%%%%%%%%%%%%%%%%%%%%%%%%%%%%%%%%%%%%%%%%%%%%%%%%%%%%
\section{The $hp$-version of the BEM and the main result} \label{sec_BEM}
\setcounter{equation}{0}

For the approximate solution of (\ref{bie_var}) we apply
the $hp$-version of the BEM on quasi-uniform triangulations of $\G$.
Our BEM is based on Galerkin discretisations with an appropriate
family of $\bH(\divg,\G)$-conforming surface elements
(here, $\bH(\divg,\G) := \{\bu \in \bL^{2}(\G);\; \divg\,\bu \in L^{2}(\G)\}$).
In what follows, $h > 0$ and $p \ge 1$ will always specify the mesh parameter and
a polynomial degree, respectively.
For any $\Omega \subset {\field{R}}^n$ we will denote
$\rho_{\Omega} = \sup \{\diam (B);\; \hbox{$B$ is a ball in $\Omega$}\}$.

Let $\CT = \{\Delta_h\}$ be a family of meshes
$\Delta_h = \{\G_j;\; j=1,\ldots,J\}$. Each mesh is a partition of $\G$
into triangular elements $\G_j$ such that $\bar\G = \cup_{j=1}^{J} \bar\G_j$,
and the intersection of any two elements $\bar\G_j,\,\bar\G_k$ ($j \not= k$)
is either a common vertex, an entire side, or empty.

In the following we always identify a face of the polyhedron $\G$
with a subdomain of ${\field{R}}^2$.
We denote $h_j=\diam (\G_j)$ for any $\G_j\in \Delta_h$.
Furthermore, any element $\G_j$ is the image
of the reference triangle
$K = \{(\xi_1,\xi_2);\; 0 < \xi_1 <1,\ 0 < \xi_2 < 1 - \xi_1\}$
under an affine mapping $T_j$, more precisely
\[
  \bar\G_j = T_j(\bar K),\quad
  \bx = T_j(\bxi) =  B_j\,\bxi + \boldb_j,
\]
where $B_j \in {\field{R}}^{2\times 2}$,
$\boldb_j \in {\field{R}}^2$, $\bx = (x_1,x_2) \in \bar\G_j$, and
$\bxi = (\xi_1,\xi_2) \in \bar K$.
Then, the Jacobian matrix of $T_j$ is $B_j \in {\field{R}}^{2\times 2}$,
and its determinant $J_j := \hbox{det}(B_j)$
satisfies the relation $|J_j| \simeq h_j^2$.

We consider a family $\CT$ of quasi-uniform shape-regular meshes $\Delta_h$
on $\G$ in the sense that there exist positive constants
$\sigma_1,\,\sigma_2$ independent of $h = \max\limits_j h_j$ such
that for any $\G_j\in\Delta_h$ and arbitrary $\Delta_h\in \CT$ there holds
\be \label{mesh}
    h_j \le \sigma_1\,\rho_{\G_j},\qquad
    h \le \sigma_2\,h_j.
\ee
Whereas the mapping $T_j$ introduced above is used to associate
scalar functions defined on the real element $\G_j$ and on the reference
triangle $K$, the Piola transformation is used to transform
vector-valued functions between $K$ and $\G_j$:
\be \label{Piola}
    \bv = \CM_j(\hat \bv) = \hbox{$\frac {1}{J_j}$} B_j \hat\bv \circ T_j^{-1},
    \quad
    \hat\bv = \CM_j^{-1}(\bv) = J_j B_j^{-1} \bv \circ T_j.
\ee

Let us introduce the needed polynomial sets.
By $\CP_p(K)$ we denote the set of polynomials of total degree $\le p$
on the reference triangle $K$.
We will use two families of $\bH(\divg,\G)$-conforming surface elements:
the Raviart-Thomas (RT) and Brezzi-Douglas-Marini (BDM) elements.
The corresponding spaces of degree $p \ge 1$ on the reference triangle $K$
will be denoted as follows (see, e.g., \cite{BrezziF_91_MHF, RobertsT_91_MHM}):
\beas
     \bCP^{\rm RT}_p(K) & = &
     (\CP_{p-1}(K))^2 \oplus \bxi \CP_{p-1}(K);
     \\[3pt]
     \bCP^{\rm BDM}_p(K) & = &
     (\CP_{p}(K))^2.
\eeas
We will use the unified notation $\bCP_p(K)$ which refers to either
the RT- or BDM-space on $K$ for $p \ge 1$.
Accordingly, all results in this paper
are formulated in a unified way and are valid for both
the RT- and BDM-based boundary element spaces defined on the triangulation of $\G$.
Note, however, that in some cases we will need to provide arguments separately for
each type of these boundary elements.

Using transformations (\ref{Piola}), we set
\be \label{Xp}
    \bX_{hp} := \{\bv \in \bH(\divg,\G);\;
                      \CM_j^{-1}(\bv|_{\G_j}) \in \bCP_p(K),\ j=1,\ldots,J\}.
\ee
Note that only one type of surface elements (i.e., either the RT- or BDM-elements)
is used in (\ref{Xp}) for all triangles $\G_j$.
We will denote by $N = N(h,p)$ the dimension of the discrete space $\bX_{hp}$.
One has $N \simeq h^{-2}$ for fixed $p$ and $N \simeq p^2$ for fixed $h$.

The $hp$-version of the Galerkin BEM for the EFIE reads as:
{\em Find $\bu_{hp}\in \bX_{hp}$ such that}
\be \label{BEM}
    a(\bu_{hp},\bv) = \<\boldf, \bv\> \quad \forall \bv \in \bX_{hp}.
\ee

Due to the infinite-dimensional kernel of the $\divg$-operator,
the bilinear form $a(\cdot,\cdot)$ in (\ref{bie_var}) is not $\bX$-coercive,
and, hence, the unique solvability of (\ref{BEM}) cannot be proved by standard
arguments. However, the refined analysis in \cite{BespalovHH_Chp}
shows that the unique BEM-solution $\bu_{hp} \in \bX_{hp}$ does exist,
and it converges quasi-optimally to the exact solution $\bu \in \bX$ of the EFIE
as $N(h,p) \to \infty$. This result is formulated in the following proposition.

\begin{prop} \label{prop_solve}
{\rm \cite[Theorem~1.2]{BespalovHH_Chp}}
There exists $N_0 \ge 1$ such that for any $\boldf \in \bX'$ and for arbitrary
mesh-degree combination satisfying $N(h,p) \ge N_0$ the discrete problem {\rm (\ref{BEM})}
is uniquely solvable and the $hp$-version of the Galerkin BEM
converges quasi-optimally, i.e.,
\be \label{quasi-optimality}
    \|\bu - \bu_{hp}\|_{\bX} \le
    C \inf\{\|\bu - \bv\|_{\bX};\; \bv\in \bX_{hp}\}.
\ee
Here, $\bu \in \bX$ is the solution of {\rm (\ref{bie_var})},
$\bu_{hp} \in \bX_{hp}$ is the solution of {\rm (\ref{BEM})},
$\|\cdot\|_{\bX}$ denotes the norm in $\bX$, and
$C>0$ is a constant independent of $h$ and $p$.
\end{prop}

The next theorem is the main result of this paper.
It states convergence rates of the $hp$-BEM with quasi-uniform meshes
for the EFIE. These convergence rates (in both the mesh parameter $h$
and the polynomial degree $p$) are given explicitly in terms of the
singularity exponents of the vector fields in (\ref{ue})--(\ref{u2ev}).

\begin{theorem} \label{thm_main_hp}
Let $\bu \in \bX$ and $\bu_{hp} \in \bX_{hp}$ be the solutions of {\rm (\ref{bie_var})}
and {\rm (\ref{BEM})}, respectively. We assume that the source functional $\boldf$
in {\rm (\ref{bie_var})} is sufficiently smooth such that representation
{\rm (\ref{dec})--(\ref{u2ev})} holds for the solution $\bu$ of {\rm (\ref{bie_var})}.
Let $v_0\in V$, $e_0\in E(v_0)$ be a vertex-edge pair such that
\[
  \min\{\lambda_1^{v_0}+1/2, \lambda_2^{v_0}+1/2, \gamma_1^{e_0}, \gamma_2^{e_0}\} =
  \min_{v\in V, e\in E(v)}\min\,\{\lambda_1^v+1/2, \lambda_2^v+1/2, \gamma_1^e, \gamma_2^e\}
\]
with $\lambda_i^v$ and $\gamma_i^e$ ($i = 1,2$) being as in {\rm (\ref{ue})--(\ref{u2ev})}.
Then for any $h > 0$ and for every
$p \ge \min\,\{\lambda_1^{v_0},\lambda_2^{v_0},\g_1^{e_0} - \frac 12,\g_2^{e_0} - \frac 12\}$
there holds
\be \label{hp_rates_1}
    \|\bu - \bu_{hp}\|_{\bX} \le
    C\,\bigg(\frac{h}{p^{\,2}}\bigg)^{\min\{\lambda_1^{v_0}+1/2,\lambda_2^{v_0}+1/2,\gamma_1^{e_0},\gamma_2^{e_0}\}}
    \Big(1 + \log\frac ph\Big)^{\beta + \nu},
\ee
where
\be \label{beta}
    \beta = 
    \cases{
           \max\,\{q_1^{v_0}+s_1^{e_0}+\frac 12,\, q_2^{v_0}+s_2^{e_0}+\frac 12\}
                       & \hbox{if \ $\lambda_i^{v_0} = \g_i^{e_0}-\frac 12$ for $i=1,2$},\cr
    \noalign{\vskip3pt}
           \max\,\{q_1^{v_0}+s_1^{e_0}+\frac 12,\, q_2^{v_0}+s_2^{e_0}\}
                       & \hbox{if \ $\lambda_1^{v_0} = \g_1^{e_0}-\frac 12$,
                                  \ $\lambda_2^{v_0} \not= \g_2^{e_0}-\frac 12$},\cr
    \noalign{\vskip3pt}
           \max\,\{q_1^{v_0}+s_1^{e_0},\, q_2^{v_0}+s_2^{e_0}+\frac 12\}
                       & \hbox{if \ $\lambda_1^{v_0} \not= \g_1^{e_0}-\frac 12$,
                                  \ $\lambda_2^{v_0} = \g_2^{e_0}-\frac 12$},\cr
    \noalign{\vskip3pt}
           \max\,\{q_1^{v_0}+s_1^{e_0},\, q_2^{v_0}+s_2^{e_0}\}
                       & \hbox{otherwise}\cr
          }
\ee
with the numbers $s_i^{e_0}$ and $q_i^{v_0}$ ($i=1,2$) given in
{\rm(\ref{ue})} and {\rm(\ref{uv})}, respectively, and
\be \label{nu}
    \nu = \cases{
                 \frac 12
                 & \hbox{if $p = \min\,\{\lambda_1^{v_0},\lambda_2^{v_0},\g_1^{e_0}-1/2,\g_2^{e_0}-1/2\}$},\cr
                 \noalign{\vskip3pt}
                 0
                 & \hbox{otherwise}.\cr
                }
\ee
If $1 \le p < \min\,\{\lambda_1^{v_0},\lambda_2^{v_0},\g_1^{e_0} - \frac 12,\g_2^{e_0} - \frac 12\}$,
then for any $h > 0$ there holds
\be \label{hp_rates_2}
    \|\bu - \bu_{hp}\|_{\bX} \le C\,h^{p+1/2}.
\ee
\end{theorem}

\begin{proof}
Considering enough singularity terms in representation (\ref{sing})
the function $\bu_{\rm reg}$ in (\ref{reg}) is as regular as needed.
Then, due to the quasi-optimal convergence (\ref{quasi-optimality})
of the $hp$-BEM with quasi-uniform meshes, the assertion follows immediately
from the general approximation result given in Theorem~\ref{thm_gen_approx} below.
\end{proof}

\begin{remark} \label{rem_parallelograms}
We have only considered meshes of triangular elements on $\G$.
If the meshes contain also shape-regular parallelogram elements (i.e., affine images
of the reference square), then a priori error estimates of
Theorem~{\rm \ref{thm_main_hp}} remain valid only in the case of the RT-based BEM.
This is because all the auxiliary results needed for the proof are
valid in this case. This, however, is not true for the BDM-based BEM.
In particular, the arguments in the proofs of Proposition~{\rm \ref{prop_solve}} above
and Proposition~{\rm \ref{prop_approx_reg}} below rely essentially
on the fact that the involved polynomial spaces
form the exact $\bcurl\,{-}\,\div$ sequence (on the reference element),
a property the BDM-spaces fail to satisfy on the reference
square (see, e.g., {\rm \cite{Demkowicz_08_PES}}).
\end{remark}

\begin{remark} \label{rem_open_surface}
We have assumed that $\G$ is a polyhedral (closed) surface.
However, all arguments in our proofs carry over only with minor modifications
to the case of a piecewise plane orientable open surface $\G$. Note that in this case
there are no restrictions needed on the wave number $\kappa$ to ensure the uniqueness
of the solution to the EFIE; the strongest edge singularities in
{\rm (\ref{ue})} have the exponents $\gamma_i^e = \frac 12$ $(i=1,2)$;
the energy space for the EFIE is $\tilde\bH^{-1/2}_0(\divg,\G)$ and
the boundary element space $\bX_{hp}$ consists of $\bH(\divg,\G)$-conforming
polynomial vector fields with normal components vanishing on $\partial\G$
(see {\rm \cite{BespalovH_NpB,BespalovH_10_hpA}} for details).
\end{remark}

%%%%%%%%%%%%%%%%%%%%%%%%%%%%%%%%%%%%%%%%%%%%%%%%%%%%%%%%%%%%%%%%%%%%%%%%%%%%%%%%
\section{General approximation result} \label{sec_gen_approx}
\setcounter{equation}{0}

In this section we prove the following general $hp$-approximation result for the
vector field $\bu$ given by formulas (\ref{dec})--(\ref{u2ev})
with the singularity exponents $\gamma_{i}^e$ and $\lambda_{i}^v$ ($i=1,2$)
satisfying the minimum requirements to guarantee $\bu \in \bX$.

\begin{theorem} \label{thm_gen_approx}
Let the vector field $\bu$ be given by {\rm (\ref{dec})--(\ref{u2ev})} on $\G$
with $\gamma_1^e,\,\gamma_2^e>0$ and $\lambda_1^v,\,\lambda_2^v>-\frac 12$
for each edge $e$ and every vertex $v$.
Also, let $v_0\in V$, $e_0\in E(v_0)$ be such that
\[
  \min\{\lambda_1^{v_0}+1/2, \lambda_2^{v_0}+1/2, \gamma_1^{e_0}, \gamma_2^{e_0}\} =
  \min_{v\in V, e\in E(v)}\min\,\{\lambda_1^v+1/2, \lambda_2^v+1/2, \gamma_1^e, \gamma_2^e\}
\]
with $\lambda_i^v$ and $\gamma_i^e$ ($i = 1,2$) being as in {\rm (\ref{ue})--(\ref{u2ev})}.
Then for any $h > 0$ and for every
$p \ge \min\,\{\lambda_1^{v_0},\lambda_2^{v_0},\g_1^{e_0} - \frac 12,\g_2^{e_0} - \frac 12\}$,
there exists $\bu^{hp} \in \bX_{hp}$ such that
\bea
     \|\bu - \bu^{hp}\|_{\bX}
     & \hskip-4pt \le \hskip-4pt &
     C\,\max\left\{
     h^{\min\{k,p\}+1/2}\,p^{-(k+1/2)},\right.
     \nonumber
     \\
     &     &
       \qquad\qquad\left.
       \bigg(\frac{h}{p^{\,2}}\bigg)^{\min\{\lambda_1^{v_0}+1/2,\lambda_2^{v_0}+1/2,\gamma_1^{e_0},\gamma_2^{e_0}\}}
       \Big(1 + \log\frac ph\Big)^{\beta + \nu}
     \right\},
     \label{gen_approx_1}
\eea
where $\beta$ and $\nu$ are defined by {\rm (\ref{beta})} and {\rm (\ref{nu})}, respectively.

If $1 \le p < \min\,\{\lambda_1^{v_0},\lambda_2^{v_0},\g_1^{e_0} - \frac 12,\g_2^{e_0} - \frac 12\}$,
then for any $h > 0$ there exists $\bu^{hp} \in \bX_{hp}$ such that
\be \label{gen_approx_2}
    \|\bu - \bu^{hp}\|_{\bX} \le C\,h^{\min\{k,p\}+1/2}.
\ee
\end{theorem}

In order to prove this theorem one needs to find discrete vector fields belonging to $\bX_{hp}$
and approximating the smooth and singular parts of $\bu$ such that the approximation errors
satisfy the upper bounds in (\ref{gen_approx_1}) and (\ref{gen_approx_2}).

We start with formulating the following $hp$-approximation result
for regular vector fields on $\G$. This result will be used, in particular, to approximate
the vector field $\bu_{\rm reg} \in \bX^k$ in (\ref{dec}).

\begin{prop} \label{prop_approx_reg}
Let $P_{hp}:\, \bX \rightarrow \bX_{hp}$ be the orthogonal projection
with respect to the norm in $\bX$. If $\bu \in \bX^k$ with
$k > 0$, then
\be \label{approx_reg}
    \|\bu - P_{hp}\bu\|_{\bX} \le
    C\,h^{\min\{k,p\}+1/2}\,p^{-(k+1/2)}\,\|\bu\|_{\bX^k}
\ee
with a positive constant $C$ independent of $h$, $p$, and $\bu$.
\end{prop}

The proof is given in \cite[Theorem~4.1]{BespalovH_10_hpA} for the case
of RT-spaces, and it carries over without
essential modifications to the case of BDM-spaces on {\em triangular}
elements (cf. Remark~\ref{rem_parallelograms}).

Now, we will study approximations of the singular part $\bu_{\rm sing}$
in representation (\ref{dec}). By (\ref{sing})--(\ref{u2ev}) and due to the arguments
in Remark~\ref{rem_sing2}~(i), we conclude that $\bu_{\rm sing}$ can be written as
\be \label{dec_sing}
    \bu_{\rm sing} = \bcurlg\, w + \bv = \bcurlg\, w + (v_1,\,v_2),
\ee
where $w \in H^{1/2}(\G)$ and $\bv \in \bH^{1/2}(\G)$.

Let us define the following discrete space (of continuous piecewise polynomials)
over the mesh $\Delta_h$:
\[
  S_{hp}(\G) := \{v\in C^0(\G);\; v|_{\G_j} \circ T_j \in \CP_p(K),\ j=1,\ldots,J\}.
\]
We will also need the following functions of $h$ and $p$:
\be \label{f_j(hp)}
    f_j(h,p) :=
    h^{\alpha_j} p^{-2 \alpha_j} (1 + \log(p/h))^{\tilde\beta_j+\nu_j},
\ee
where $j = 1,2$,
\be \label{alpha_j}
    \alpha_j := \min\,\{\lambda_j^{v_0}+1/2,\g_j^{e_0}\},
\ee
\be \label{beta_j}
    \tilde\beta_j := \cases{
                            q_j^{v_0} + s_j^{e_0} +\frac 12
                            & \hbox{if $\lambda_j^{v_0} = \g_j^{e_0} - \frac 12$},\cr
                     \noalign{\vskip3pt}
                            q_j^{v_0} + s_j^{e_0}
                            & \hbox{otherwise},\cr
                           }
\ee
\be \label{nu_j}
    \nu_j := \cases{
                    \frac 12
                    & \hbox{if $p = \alpha_j - \frac 12$},\cr
                     \noalign{\vskip3pt}
                    0
                    & \hbox{otherwise},\cr
                   }
\ee
and the numbers $\g_j^{e_0},\, \lambda_j^{v_0},\, s_j^{e_0},\, q_j^{v_0}$ ($j = 1,2$)
are given in (\ref{ue})--(\ref{u2ev}) for the vertex-edge pair $(v_0,e_0)$
introduced in the formulation of Theorem~\ref{thm_gen_approx}.

In the next two lemmas we formulate approximation results for the
vector fields $\bcurlg\, w$ and $\bv$ on the right-hand side of (\ref{dec_sing}).
The proofs are given in Section~\ref{sec_proofs} below.

\begin{lemma} \label{lm_approx_w}
Let $w \in H^{1/2}(\G)$ be the scalar singular function in representation {\rm (\ref{dec_sing})}.
Then for any $h > 0$ and $p \ge 1$, there exists $w^{hp} \in S_{hp}(\G)$
such that $\bcurlg\, w^{hp} \in \bX_{hp}$ and there holds
\be \label{approx_w}
    \|\bcurlg\,w - \bcurlg\,w^{hp}\|_{\bX} \le
    \cases{
           C\,f_1(h,p)
           & \hbox{if \ $p \ge \alpha_1 - \frac 12$},\cr
           \noalign{\vskip3pt}
           C\,h^{p+1/2}
           & \hbox{if \ $1 \le p < \alpha_1 - \frac 12$},\cr
          }
\ee
where $f_1(h,p)$ and $\alpha_1$ are defined by {\rm (\ref{f_j(hp)})}
and {\rm (\ref{alpha_j})}, respectively.
\end{lemma}

\begin{lemma} \label{lm_approx_v}
Let $\bv \in \bH^{1/2}(\G)$ be the singular vector field in representation {\rm (\ref{dec_sing})}.
Then for any $h > 0$ and $p \ge 1$, there exists $\bv^{hp} \in \bX_{hp}$
satisfying
\be \label{approx_v}
    \|\bv - \bv^{hp}\|_{\bX} \le
    \cases{
           C\,f_2(h,p)
           & \hbox{if \ $p \ge \alpha_2 - \frac 12$},\cr
           \noalign{\vskip3pt}
           C\,h^{p+1/2}
           & \hbox{if \ $1 \le p < \alpha_2 - \frac 12$},\cr
          }
\ee
where $f_2(h,p)$ and $\alpha_2$ are defined by {\rm (\ref{f_j(hp)})}
and {\rm (\ref{alpha_j})}, respectively.
\end{lemma}

Now we are able to prove Theorem~\ref{thm_gen_approx}.
\medskip

\noindent{\bf Proof of Theorem~\ref{thm_gen_approx}.}
For the regular vector field $\bu_{\rm reg} \in \bX^k$ in (\ref{dec}) we use the
orthogonal projection $P_{hp}:\,\bX \rightarrow \bX_{hp}$
with respect to the norm in $\bX$ to define
$\bu_{\rm reg}^{hp} := P_{hp} \bu_{\rm reg} \in \bX_{hp}$.
Then we have by Proposition~\ref{prop_approx_reg}
\be \label{approx_reg_1}
    \|\bu_{\rm reg} - \bu_{\rm reg}^{hp}\|_{\bX} \le
    C\,h^{\min\{k,p\}+1/2}\,p^{-(k+1/2)}.
\ee
Since the singular part of decomposition (\ref{dec}) can be written as in (\ref{dec_sing}),
we use approximations $w^{hp} \in S_{hp}(\G)$ and $\bv^{hp} \in \bX_{hp}$
from the above two lemmas to define
$\bu_{\rm sing}^{hp} := \bcurlg\, w^{hp} + \bv^{hp} \in \bX_{hp}$.
Then, applying the triangle inequality, we obtain by (\ref{approx_w}) and (\ref{approx_v})
\be \label{approx_sing}
    \|\bu_{\rm sing} - \bu_{\rm sing}^{hp}\|_{\bX} \le
    \cases{
           C\,\max \left\{f_1(h,p),\;f_2(h,p)\right\}
           & \hbox{if \ $p \ge \min\,\{\alpha_1,\alpha_2\} - \frac 12$},\cr
           \noalign{\vskip3pt}
           C\,h^{p+1/2}
           & \hbox{if \ $1 \le p < \min\,\{\alpha_1,\alpha_2\} - \frac 12$}.\cr
          }
\ee
Now, we set $\bu^{hp} := \bu_{\rm reg}^{hp} + \bu_{\rm sing}^{hp} \in \bX_{hp}$.
Combining estimates (\ref{approx_reg_1}) and (\ref{approx_sing}), 
applying the triangle inequality, and using expressions
(\ref{f_j(hp)}) and (\ref{alpha_j}) for the functions $f_j(h,p)$
and the parameters $\alpha_j$, respectively, we prove the desired estimates in
(\ref{gen_approx_1}) and (\ref{gen_approx_2}).\qed

%%%%%%%%%%%%%%%%%%%%%%%%%%%%%%%%%%%%%%%%%%%%%%%%%%%%%%%%%%%%%%%%%%%%%%%%%%%%%%%%
\section{Proofs of technical lemmas} \label{sec_proofs}
\setcounter{equation}{0}

In this section we prove Lemmas~\ref{lm_approx_w} and~\ref{lm_approx_v}.

%%%%%%%%%%%%%%%%%%%%%%%%%%%%%%%%%%%%%%%%%%%%%%%%%%%%%%%%%%%%%%%%%%%%%%%%%%%%%%%%
\subsection{Proof of Lemma~\ref{lm_approx_w}} \label{sec_proof_w}

Recalling Remark~\ref{rem_sing2}~(ii), we use 
the results of \cite{BespalovH_08_hpB,Bespalov_09_NPA}
to find the desired piecewise polynomial $w^{hp} \in S_{hp}(\G)$
such that the norm $\|w - w^{hp}\|_{H^{1/2}(\G)}$ is bounded as
in (\ref{approx_w}) (see Theorems~5.1,~5.2 in \cite{BespalovH_08_hpB}
and Theorem~4.1 in \cite{Bespalov_09_NPA}). Then, recalling the fact that the
the operator $\bcurlg: H^{1/2}(\G) \rightarrow \bH^{-1/2}_{\|}(\G)$
is continuous (see \cite{BuffaC_01_TII}), we derive the estimate in (\ref{approx_w}):
\beas
     \|\bcurlg\,w - \bcurlg\,w^{hp}\|_{\bX}
     & = &
     \|\bcurlg (w - w^{hp})\|_{\bH^{-1/2}_{\|}(\G)}
     \\[3pt]
     & \le &
     C \|w - w^{hp}\|_{H^{1/2}(\G)} \le
    \cases{
           C\,f_1(h,p)
           & \hbox{if \ $p \ge \alpha_1 - \frac 12$},\cr
           \noalign{\vskip3pt}
           C\,h^{p+1/2}
           & \hbox{if \ $1 \le p < \alpha_1 - \frac 12$}.\cr
          }
\eeas
It remains to prove that $\bcurlg\, w^{hp} \in \bX_{hp}$.
In fact, it is easy to check (see~\cite[p.~615]{BespalovH_NpB}) that
\[
     \CM_j^{-1}(\bcurlg\, w^{hp}|_{\G_j}) =
     \bcurl\, (w^{hp}|_{\G_j} \circ T_j),\qquad
     \hbox{where
           $\bcurl = \Big(
                          \hbox{$\frac{\partial}{\partial \xi_{2}}$},\,
                          -\hbox{$\frac{\partial}{\partial \xi_{1}}$}
                     \Big)$.
          }
\]
Hence, recalling that $w^{hp}|_{\G_j} \circ T_j \in \CP_p(K)$,
we conclude that
$\CM_j^{-1}(\bcurlg\, w^{hp}|_{\G_j}) \in
 (\CP_{p-1}(K))^2 \subset \bCP_p(K)$
(this is because $(\CP_{p-1}(K))^2$ is a subset of both $\bCP^{\rm BDM}_p(K)$ and
$\bCP^{\rm RT}_p(K)$).
Moreover, $\bcurlg\, w^{hp} \in \bH(\divg,\G)$,
because $\divg(\bcurlg\, w^{hp}) \equiv 0$ on $\G$.
Therefore, $\bcurlg\, w^{hp} \in \bX_{hp}$, and the proof is finished.

%%%%%%%%%%%%%%%%%%%%%%%%%%%%%%%%%%%%%%%%%%%%%%%%%%%%%%%%%%%%%%%%%%%%%%%%%%%%%%%%
\subsection{Proof of Lemma~\ref{lm_approx_v}} \label{sec_proof_v}

Let $\bv = (v_1,v_2)$ be the second term in decomposition (\ref{dec_sing}) of the
singular vector field $\bu_{\rm sing}$. Again, using the results
in \cite{BespalovH_08_hpB,Bespalov_09_NPA}, we find continuous piecewise polynomial
approximations to the scalar components $v_1,\ v_2$ of $\bv$:
for any $h > 0$ and every $p \ge 1$,
there exist $v_1^{hp},\,v_2^{hp} \in S_{hp}(\G)$ such that
for $i = 1,2$ there holds
\be \label{v_approx1}
    \|v_i - v_{i}^{hp}\|_{H^{1/2}(\G)} \le
    \cases{
           C\,f_2(h,p)
           & \hbox{if \ $p \ge \alpha_2 - \frac 12$},\cr
           \noalign{\vskip3pt}
           C\,h^{p+1/2}
           & \hbox{if \ $1 \le p < \alpha_2 - \frac 12$}\cr
          }
\ee
with positive constants $C > 0$ independent of $h$ and $p$.

Let $\bv^{hp} = (v_1^{hp},\,v_2^{hp})$. Then $\bv^{hp} \in \bH(\divg,\G)$.
Furthermore, in the case of BDM-elements we observe that
for any $\G_j$ there holds
\[
  \CM_j^{-1}(\bv^{hp}|_{\G_j}) =
  J_j B_j^{-1} (\bv^{hp}|_{\G_j}) \circ T_j \in
  (\CP_{p}(K))^2 = \bCP^{\rm BDM}_p(K).
\]
Therefore, $\bv^{hp} \in \bX_{hp}$ in this case.
Moreover, since $\bv \in \bH^{1/2}(\G)$ and $v^{hp}_i \in S_{hp}(\G)$,
we use estimate (\ref{v_approx1}) and the continuity of the operator
$\divg: \bH^{1/2}(\G) \rightarrow H^{-1/2}(\G)$ to obtain
\bea
    \|\bv - \bv^{hp}\|_{\bX}
    & = &
    \|\bv - \bv^{hp}\|_{\bH^{-1/2}_{\|}(\G)} +
    \|\divg (\bv - \bv^{hp})\|_{H^{-1/2}(\G)}
    \nonumber
    \\
    & \le &
    C \|\bv - \bv^{hp}\|_{\bH^{1/2}(\G)} \le
    C \sum\limits_{i=1}^{2} \|v_i - v_i^{hp}\|_{H^{1/2}(\G)}
    \nonumber
    \\
    & \le &
    \cases{
           C\,f_2(h,p)
           & \hbox{if \ $p \ge \alpha_2 - \frac 12$},\cr
           \noalign{\vskip3pt}
           C\,h^{p+1/2}
           & \hbox{if \ $1 \le p < \alpha_2 - \frac 12$}.\cr
          }
    \label{v_approx2}
\eea

Unfortunately, in the case of RT-elements this component-wise approximation
of $\bv$ does not work since the dimension of the RT-space on the reference triangle
is smaller than the dimension of the BDM-space, and, in general,
$\CM_j^{-1}(\bv^{hp}|_{\G_j}) \notin \bCP^{\rm RT}_p(K)$, so that
$\bv^{hp} = (v_1^{hp},\,v_2^{hp}) \notin \bX_{hp}$.
In this case we follow the procedure described in \cite{BespalovH_08_hpB} (for the scalar case).
More precisely, we use appropriate
$h$-scaled cut-off functions and represent the vector field $\bv$ as the sum of
a singular vector field $\bphi$ with small support in the vicinity of the edges
and a sufficiently smooth vector field $\bpsi$.
In the rest of this subsection we will demonstrate
how this procedure works in the vector case. We will give a concise step-by-step
outline of the procedure referring frequently to \cite{BespalovH_08_hpB}
for particular error estimates and other technical details.

{\bf Step 1: decomposition of $\bv$.}
Let $h_0 = (\sigma_1 \sigma_2)^{-1} h$ with $\sigma_1,\,\sigma_2$ from
(\ref{mesh}). Using the scaled cut-off functions
$\chi^e_2(x_{e2}/h_0)$ and $\chi^v(r_v/h_0)$ with $\chi^e_2$ and $\chi^v$
from (\ref{ue}) and (\ref{uv}), respectively, one can decompose $\bv$
as follows (cf. \cite[eqs. (5.4), (5.21), (6.4)]{BespalovH_08_hpB})
\be \label{v_approx3}
    \bv = \bphi + \bpsi,
\ee
where $\supp\,\bphi \subset \cup_{v \in V} \cup_{e \in E(v)} (\bar A_e \cup \bar A_v)$,
$\bphi \in \bH^{1/2}(\G)$, $\bpsi \in \bX^m$ (see (\ref{reg}) for the notation), and
$\bpsi$ vanishes in small ($h$-dependent) neighbourhoods
of each vertex and each edge of~$\G$. Here, $A_e$ is the union of elements at one edge
$e$, i.e., $\bar A_{e} := \cup\{\bar\G_j;\; \bar\G_j \cap e \not= \mbox{\o}\}$
(note that the endpoints of $e$ are not included in $e$),
$A_v$ is the union of elements at a vertex $v$, i.e.,
$\bar A_{v} := \cup\{\bar\G_j;\; v \in \bar\G_j\}$,
$m$ is sufficiently large and depends
on the parameters $t_2$ and $m_2$ specified for the singularities $\bu^v$
and $\bu_2^{ev}$, respectively.

{\bf Step 2: approximation of $\bphi$ for $p = 1$.}
If $p = 1$, then one can approximate $\bphi$ by zero.
One has $\bphi_{hp} \equiv \bzero \in \bX_{hp}$, and, recalling the continuity
of the operator $\divg$, we derive
\be \label{v_approx4}
    \|\bphi - \bphi_{hp}\|_{\bX} =
    \|\bphi\|_{\bH^{-1/2}_{\|}(\G)} + \|\divg\,\bphi\|_{H^{-1/2}(\G)} \le
    C \|\bphi\|_{\bH^{1/2}(\G)}.
\ee
We will obtain $h$-estimates for the norms of $\bphi$ in $\bH^s(\G)$
with $s=0$ and $s = \frac 12 + \eps$ (for sufficiently small $\eps > 0$)
by using the fact that $\bphi$ has a small support.
First, we apply Lemma~3.1 of \cite{BespalovH_05_pBE} and Lemma~3.5 of \cite{BespalovH_08_hpB}
to localise the norm $\|\bphi\|_{\bH^{1/2+\eps(\G)}}$ to the faces $\G^f \subset \G$
and to the elements $\G_j \subset \G^f$, respectively. Then, we use scaling
on each element $\G_j \subset \supp\,\bphi$
(cf. \cite[Lemma~3.1, eqs. (5.6), (5.11), (5.23), (6.7)]{BespalovH_08_hpB}).
As a result we have
\beas
     \|\bphi\|^2_{\bH^s(\G)}
     & \le &
     C \sum\limits_{f:\, \G^f \subset \G} \|\bphi\|^2_{\bH^s(\G^f)}
     \\
     & \le &
     C \sum\limits_{f:\, \G^f \subset \G}\
       \sum\limits_{j:\, \G_j \subset \G^f}
       \bigg(h^{-2s} \|\bphi\|^2_{\bL^2(\G_j)} + |\bphi|^2_{\bH^s(\G_j)}\bigg)
     \\[3pt]
     & \le &
     C h^{2(\alpha_2 + 1/2 - s)} (1+\log(1/h))^{2\tilde\beta_2}\qquad
     \hbox{for $s \in \{0, 1/2 + \eps\}$}.
\eeas
Here, $\alpha_2$ and $\tilde\beta_2$ are defined  by (\ref{alpha_j})
and (\ref{beta_j}), respectively.
Hence, using the interpolation between $\bH^0(\G)$ and $\bH^{1/2+\eps}(\G)$,
we obtain by (\ref{v_approx4})
\be \label{v_approx5}
    \|\bphi - \bphi_{hp}\|_{\bX} \le
     C h^{\alpha_2} (1+\log(1/h))^{\tilde\beta_2}.
\ee

{\bf Step 3: approximation of $\bphi$ for $p \ge 2$.}
We approximate each component $\varphi_i$ of $\bphi$ by a piecewise
polynomial $\varphi_i^{hp} \in S_{h,p-1}(\G)$ ($i = 1,2$).
Here we can use the results of \cite{BespalovH_08_hpB,Bespalov_09_NPA} for
each type of singularity: there exist $\varphi_i^{hp} \in S_{h,p-1}(\G)$ such that
for $i = 1,2$ there holds (cf. \cite[eqs. (5.12), (5.13), (5.24), (6.9)]{BespalovH_08_hpB})
\be \label{v_approx6}
    \|\varphi_i - \varphi_{i}^{hp}\|_{H^{1/2}(\G)} \le
    C\,h^{\alpha_2}\,(p-1)^{-2\alpha_2}\,(1+\log\hbox{$\frac{p-1}{h}$})^{\tilde\beta_2} \le
    C\,h^{\alpha_2}\,p^{-2\alpha_2}\,(1+\log\hbox{$\frac{p}{h}$})^{\tilde\beta_2},
\ee
where $\alpha_2$ and $\tilde\beta_2$ are the same as in (\ref{v_approx5}).

Setting $\bphi^{hp} = (\varphi_1^{hp},\varphi_2^{hp})$ we observe that
$\bphi^{hp} \in \bH(\divg,\G)$, and for any element $\G_j$ one has
\[
  \CM_j^{-1}(\bphi^{hp}|_{\G_j}) =
  J_j B_j^{-1} (\bphi^{hp}|_{\G_j}) \circ T_j \in
  (\CP_{p-1}(K))^2 \subset \bCP^{\rm RT}_p(K).
\]
Hence, $\bphi^{hp} \in \bX_{hp}$. Moreover, since $\bphi \in \bH^{1/2}(\G)$
and $\varphi_i^{hp} \in S_{h,p-1}(\G)$ for $i=1,2$,
we estimate by analogy with (\ref{v_approx4}) and using (\ref{v_approx6}):
\bea \label{v_approx7}
     \|\bphi - \bphi_{hp}\|_{\bX}
     & \le &
     C \|\bphi - \bphi_{hp}\|_{\bH^{1/2}(\G)}
     \nonumber
     \\
     & \le &
     C \sum\limits_{i=1}^{2} \|\varphi_i - \varphi_i^{hp}\|_{H^{1/2}(\G)} \le
     C\,h^{\alpha_2}\,p^{-2\alpha_2}\,(1+\log\hbox{$\frac{p}{h}$})^{\tilde\beta_2}.
\eea
Comparing (\ref{v_approx5}) and (\ref{v_approx7}) we observe that
one can use estimate (\ref{v_approx7}) also in the case $p = 1$.

{\bf Step 4: approximation of $\bpsi$.}
Recalling that $\bpsi \in \bX^m$ we apply Proposition~\ref{prop_approx_reg}:
there exists $\bpsi^{hp} \in \bX_{hp}$ such that
\be \label{v_approx8}
    \|\bpsi - \bpsi^{hp}\|_{\bX} \le
    C\,h^{\min\{k,p\}+1/2}\,p^{-(k+1/2)}\,\|\bpsi\|_{\bX^k},\qquad
    0 \le k \le m.
\ee
We also recall that $\bpsi$ vanishes in the $h$-neighbourhood
of all edges and vertices of $\G$. Therefore, having explicit expressions
of the singular vector field $\bv$ (and thus, of $\bpsi$) on each face $\G^f \subset \G$,
one can derive upper bounds for the norms $\|\bpsi\|_{\bH^k(\G^f)}$ and
$\|\div_{\G^f}\bpsi\|_{H^k(\G^f)}$ in terms of the mesh size $h$, the parameter $k$,
and the singularity exponents.
For the norm of $\bpsi$ this can be done component-wise using the same calculations
as in \cite{BespalovH_08_hpB} (see, e.g., inequalities (5.15) and (6.10) therein):
\be \label{v_approx9}
    \|\bpsi\|_{\bH^k(\G^f)} \le
    C\,h^{\alpha_2 + 1/2 - k}\,(\log(1/h))^{\tilde\beta_2+\tilde\nu_2},\qquad
    \alpha_2 + \hbox{$\frac 12$} \le k \le m,
\ee
where $\alpha_2$ and $\tilde\beta_2$ are defined by (\ref{alpha_j}) and (\ref{beta_j}),
respectively, $\tilde\nu_2 = \frac 12$ if $k = \alpha_2 + \frac 12$ and
$\tilde\nu_2 = 0$ otherwise.

To estimate the norm $\|\div_{\G^f}\bpsi\|_{H^k(\G^f)}$ we observe that the operator
$\div_{\G^f}$ reduces all singularity exponents by one (while preserving the structure
of the corresponding singularity). Then, using similar calculations as indicated above,
we obtain
\be \label{v_approx10}
    \|\div_{\G^f}\,\bpsi\|_{H^k(\G^f)} \le
    C\,h^{\alpha_2 - 1/2 - k}\,(\log(1/h))^{\tilde\beta_2+\bar\nu_2},\qquad
    \alpha_2 - \hbox{$\frac 12$} \le k \le m,
\ee
where $\alpha_2,\,\tilde\beta_2$ are the same as in (\ref{v_approx9}), whereas
$\bar\nu_2 = \frac 12$ if $k = \alpha_2 - \frac 12$ and
$\bar\nu_2 = 0$ otherwise.

By (\ref{v_approx8})-(\ref{v_approx10}) we conclude that
\be \label{v_approx11}
    \|\bpsi - \bpsi^{hp}\|_{\bX} \le
    C\,h^{\min\{k,p\}+\alpha_2-k}\,p^{-(k+1/2)}\,(\log(1/h))^{\tilde\beta_2+\bar\nu_2},\qquad
    \alpha_2 - \hbox{$\frac 12$} \le k \le m
\ee
with the same $\alpha_2$, $\tilde\beta_2$, and $\bar\nu_2$ as in (\ref{v_approx10}).

Let $p > 2\alpha_2 - \frac 12$. Since $m$ is large enough, we can select an integer $k$ satisfying
\[
  2\alpha_2 - \hbox{$\frac 12$} < k \le \min\,\{m,p\}.
\]
Then $\min\,\{k,p\} = k$, and $p^{-(k+1/2)} \le p^{-2\alpha_2}$.

If $\alpha_2 - \frac 12 < p \le 2\alpha_2 - \frac 12$ (i.e., $p$ is bounded),
we choose an integer $k \in (\alpha_2 - \frac 12,p]$,
and if $p = \alpha_2 - \frac 12$, then we take $k = p = \alpha_2 - \frac 12$.
In both these cases $\min\,\{k,p\} = k$, and $p^{-(k+1/2)} \le C\,p^{-2\alpha_2}$.

Thus, for any $p \ge \alpha_2 - \frac 12$, selecting $k$ as indicated above
we find by (\ref{v_approx11})
\be \label{v_approx12}
    \|\bpsi - \bpsi^{hp}\|_{\bX} \le
    C\,h^{\alpha_2}\,p^{-2\alpha_2}\,(\log(1/h))^{\tilde\beta_2+\nu_2}
\ee
with $\alpha_2$, $\tilde\beta_2$, and $\nu_2$ being defined by
(\ref{alpha_j}), (\ref{beta_j}), and (\ref{nu_j}), respectively.

{\bf Step 5: approximation of $\bv = \bphi + \bpsi$.}
Let us define $\bv^{hp} := \bphi^{hp} + \bpsi^{hp} \in \bX_{hp}$, where
$\bphi^{hp}$ and $\bpsi^{hp}$ are approximations constructed above (see Steps~2--4).
Then combining estimates (\ref{v_approx7}) and (\ref{v_approx12}) and
using the triangle inequality we prove (\ref{approx_v}) in the case
$p \ge \alpha_2 - \frac 12$.

It remains to consider the case $1 \le p < \alpha_2 - \frac 12$.
In this case one does not need decomposition (\ref{v_approx3}).
Observe that for each face $\G^f$ one has $\div_{\G^f}\,\bv \in H^k(\G^f)$
with $1 \le k < \alpha -\frac 12$. Therefore, $\bv \in \bX^k$
with $1 \le k < \alpha -\frac 12$, and applying Proposition~\ref{prop_approx_reg}
we find $\bv^{hp} \in \bX_{hp}$ satisfying
\[
  \|\bv - \bv^{hp}\|_{\bX} \le
  C\,h^{\min\,\{k,p\}+1/2}\,\|\bv\|_{\bX^k}.
\]
Hence, selecting $k \in [p, \alpha - \frac 12)$ we arrive at the desired
upper bound in (\ref{approx_v}), and the proof is finished.

%%%%%%%%%%%%%%%%%%%%%%%%%%%%%%%%%%%%%%%%%%%%%%%%%%%%%%%%%%%%%%%%%%%%%%%%%%%%%%%%
%*flatex input: [BespalovH_Nhp.bbl]

% flatex input end: [BespalovH_Nhp.bbl]
%FLATEX-REM:\bibliographystyle{siam}
%FLATEX-REM:\bibliography{bib,heuer,fem}

\end{document}